\documentclass[12pt]{amsart}
\usepackage{amssymb}
\usepackage{graphicx}
\usepackage{hyperref}
\usepackage[all]{xy}

\usepackage{fullpage}

\fontsize{12}{18}\selectfont
\usepackage{amsmath}
\usepackage{amssymb}

\DeclareMathOperator{\Ind}{Ind}
\newtheorem{thm}{Theorem}
\newtheorem{prop}[thm]{Proposition}
\newtheorem{la}[thm]{Lemma}

\newtheorem{cor}[thm]{Corollary}

\theoremstyle{definition}

\theoremstyle{remark}

\numberwithin{equation}{section}
\def \P {{\mathbb P}}
\def \A {{\mathbb A}}
\def \E {{\mathcal E}}
\def \Z {{\mathbb Z}}
\def \rk {{\rm rk}}

\begin{document}
\title{Embedding problems and open subgroups}
\author{David Harbater and Katherine Stevenson}

\email{harbater@math.upenn.edu, katherine.stevenson@csun.edu}
\thanks{The authors were respectively supported in part by NSF grants 
DMS-0901164 and IIS-0534984.}

\date{December 6, 2009.
}

\maketitle

 \begin{abstract} We study the properties of the fundamental group of an affine curve over an algebraically closed field of characteristic $p$, from the point of view of embedding problems.  In characteristic zero, the fundamental group is free, but in characteristic $p$ it is not even $\omega$-free.  In this paper we show that it is ``almost $\omega$-free,'' in the sense that each finite embedding problem has a proper solution when restricted to some open subgroup.  We also prove that embedding problems can always be properly solved over the given curve if suitably many additional branch points are allowed, in locations that can be specified arbitrarily; this strengthens a result of the first author. \end{abstract}

\maketitle
\date{today}

\section{Introduction.}

This paper studies the structure of the \'etale fundamental group of an affine curve over an algebraically closed field of characteristic $p$, and shows that in a certain sense it is ``not too far from being free.''

Consider an affine curve $C$ over an algebraically closed field $k$.  Thus $C = X - \Sigma$, where $X$ is a smooth projective curve of genus $g \ge 0$ and $|\Sigma| = r > 0$.  If $k$ has characteristic zero, then $\pi_1(C)$ is a free profinite group on $2g+r-1$ generators \cite[XIII, Cor.~2.12]{Gr}.  In contrast, if $k$ has characteristic $p > 0$, then $\pi_1(C)$ is infinitely generated as a profinite group, since by Artin-Schreier theory there are infinitely many linearly disjoint $\Z/p\Z$-Galois \'etale covers of $C$.  Moreover it cannot be free profinite (of infinite rank), since a prime-to-$p$ group is a quotient of $\pi_1(C)$ if and only if it has at most $2g+r-1$ generators \cite[XIII, Cor.~2.12]{Gr}).  For this same reason, $\pi_1(C)$ is not even $\omega$-free; i.e.\ it is not the case that every finite embedding problem for $\pi_1(C)$ has a proper solution \cite[p.~652]{FJ}.  (Infinitely generated free profinite groups are $\omega$-free \cite[Lemma~25.1.2]{FJ}; the converse holds for countably generated groups \cite[p.~167]{Iw}, but not in general \cite[Example~3.1]{Ja}.)

The main result of this paper (Theorem~\ref{app1}) is to show that $\Pi = \pi_1(C)$ has the weaker property of being \textit{almost $\omega$-free}: for every finite embedding problem $\mathcal{E} = (\alpha : \Pi \to G, f: \Gamma \to G)$ there is an open subgroup $H \subset \Pi$ 
such that $\alpha|_H: H \to G$ is surjective and the induced embedding problem $\mathcal{E}_H = (\alpha|_H, f)$ has a proper solution.

This result can be interpreted geometrically as follows.  
We are given a $G$-Galois cover $\phi:Y \to X$ ramified at most over $\Sigma$.  
A proper solution to the original embedding problem is a $\Gamma$-Galois cover $Z \to X$ that dominates $\phi: Y \to X$ and is unramified away from $\Sigma$.  The almost $\omega$-free
property asserts that there exists a cover $\psi: X' \to X$ ramified at most over $\Sigma$ such that the normalized pullback $\phi': Y' \to X'$ of $\phi$ via $\psi$ is a $G$-Galois cover and moreover there exists a $\Gamma$-Galois cover $Z' \to X'$ that dominates $Y' \to X'$ and is branched only at points of $X'$ lying over $\Sigma \subset X$.  Here the cover $\psi$ corresponds to the restriction $\alpha|_H$, where $H \subset \pi_1(X - \Sigma)$ is the open subgroup $\pi_1(X' - \psi^{-1}(\Sigma))$.

This paper is motivated by the fact that the absolute Galois group $G_{k(X)}$ of the function field of a smooth projective $k$-curve $X$ is free, for $k$ an algebraically closed field of arbitrary characteristic (\cite{Ha95}, \cite{Po}).  In characteristic zero this is natural since each $\pi_1(C)$ is free; and in that case the freeness of $G_{k(X)}$ was proven earlier by viewing this group as the inverse limit of the free groups $\pi_1(C)$ for dense affine open subsets $C \subset X$ \cite{Do}.  But in non-zero characteristic, the freeness of $G_{k(X)}$ is a bit mysterious, since the groups $\pi_1(C)$ are not free.  Still, this suggests that the groups $\pi_1(C)$ should satisfy some kind of freeness property; and this role is played here by the property of being almost $\omega$-free.  

In the process of proving our main theorem, we also obtain a result (Theorem~\ref{solveEP}) that strengthens the main theorem of \cite{Ha99}.  That theorem, which provided a more precise form of a key result in \cite{Ha95} and \cite{Po}, had asserted that an embedding problem for an affine curve over an algebraically closed field can be solved if a specific number of additional points are permitted.  Our result here shows that the positions of the additional points can be chosen arbitrarily. 

We also note the papers \cite{ku08} and \cite{ku09}, which discuss another way in which $\pi_1(C)$ is ``close to being free''.  Namely, those papers show that if $k$ is an algebraically closed field of characteristic $p$ then the commutator subgroup of $\pi_1(C)$ is free.  For 
other papers on related themes, see also \cite{Bo},  \cite{GSt}, \cite{Os}, \cite{Pr}, \cite{Ta}.

The structure of the current manuscript is as follows:  In Section~\ref{newram} we prove Proposition~\ref{spts}, which asserts that $G$-Galois branched covers have pullbacks that remain $G$-Galois but become unramified at sufficiently many points lying over the original branch locus.  Section~\ref{solve} proves results leading up to Theorem~\ref{solveEP} on the solution to embedding problems with prescribed additional branch points.  These are combined in Section~\ref{application} to obtain our main result, Theorem~\ref{app1}, that the fundamental group is almost $\omega$-free.  That section also discusses almost $\omega$-freeness in the characteristic zero case, where the fundamental group is finitely generated.

\medskip

{\it Terminology and Notation:}
If $G$ is a finite group and $p$ is a prime number, let $p(G)$ denote the subgroup of $G$ generated by its $p$-subgroups.  
This is a characteristic subgroup of $G$, and $G/p(G)$ is the maximal prime-to-$p$ quotient of $G$. 
A finite group $G$ is called {\it quasi-$p$} if $G = p(G)$.
A {\it finite embedding problem} $\E$ for a group $\Pi$ is a pair of surjections $(\alpha:\Pi \to G, f:\Gamma \to G)$, where $\Gamma$ and $G$ are finite groups.
A {\it weak solution} to $\mathcal{E}$ is a homomorphism $\gamma: \Pi \to \Gamma$ such that $f \circ \gamma = \alpha$.  We call $\gamma$
a {\it proper solution} to $\mathcal E$ if in addition it is surjective.
If $H$ is a subgroup of $\Pi$ and $\alpha|_{H}$ is a surjection onto $G$, then $\E' = (\alpha|_{H}:H \to G,  f:\Gamma \to G)$ is the \textit{induced} embedding problem for $H$.

In this paper we consider curves over an algebraically closed field $k$.  A {\it cover} of $k$-curves is a morphism $\phi:V \to U$ of smooth connected $k$-curves that is finite and generically separable.  
If $\phi:V \to U$ is a cover, its \emph{Galois group} $\mathrm{Gal}(V/U)$ is the group of $k$-automorphisms 
$\sigma$ of $V$ satisfying $\phi \circ \sigma = \phi$.  
If $G$ is a finite group, then a $G$-\emph{Galois cover} is a cover $\phi: V \to U$ 
\emph{together} with an inclusion $\rho: G \hookrightarrow \mathrm{Gal}(V/U)$ 
such that $G$ acts simply transitively on a generic geometric fibre of $\phi:V \to U$.  If we fix a base point of $U$, then the pointed $G$-Galois \'etale covers of $U$ correspond bijectively to the surjections $\alpha: \pi_1(U) \to G$, where $\pi_1(U)$ is the algebraic fundamental group of $U$.  The proper solutions to an embedding problem $\E = (\alpha:\pi_1(U) \to G, f:\Gamma \to G)$ for $\pi_1(U)$ then are in bijection to the pointed $\Gamma$-Galois covers $W \to U$ that dominate the pointed $G$-Galois cover $\phi: V \to U$ corresponding to $\alpha$.

\medskip

{\it Acknowledgment.} We thank Florian Pop for suggestions involving Section~3 of this paper.

\section{Removing branch points by pulling back.} \label{newram}

Fix an algebraically closed field $k$ of characteristic $p>0$.  
The goal of this section is to prove Proposition~\ref{spts} below, which says that for any branched cover of $k$-curves, there is a finite morphism to the base such that the inverse image of the original branch locus contains (many) points that are unramified in the normalized pullback of the given cover.  This will be used in conjunction with the key result of the next section in order to prove our main theorem.  The point is that since affine curves are not $\omega$-free, not all embedding problems can be properly solved; but by proceeding to a pullback via Proposition~\ref{spts}, we may enlarge the branch locus and thereby obtain a solution to the induced embedding problem.

\begin{la}\label{sizeSigma}
Let $X$ be a smooth connected projective $k$-curve and let $\Delta \subset X$ be a non-empty finite subset of $X$.   Given an integer $t>1$ and a $G$-Galois cover $\phi:Y \to X$ \'etale away from $\Delta$, there is a Galois cover $\pi:X_0 \to X$ of degree at least $t$ that is linearly disjoint from $Y \to X$, is branched only over $\Delta$, and such that $|\pi^{-1}(\Delta)| \ge t$.
\end{la}
\noindent{\it Proof.}
Fix a point $\sigma \in \Delta$, and choose an integer $n \ge {\rm max}(5,p)$ such that $A_n$ is not a quotient of $G$ and the largest prime-to-$p$ factor $m$ of $|A_n|$ is at least $t$.  The group $A_n$ is simple, so every $A_n$-Galois cover of $X$ is linearly disjoint from $Y$ over $X$, having no common subcovers.  

Since $A_n$ is simple of order divisible by $p$, it is a quasi-$p$ group.  So by Abhyankar's Conjecture \cite[Theorem~6.2]{Ha94} (or by \cite[Corollaire~2.2.2]{Ra}, in the case of the affine line), there is an $A_n$-Galois cover $\pi:X_0 \to X$ that is branched only at $\sigma \in X$.  Moreover, this cover may be chosen so that the inertia groups over $\sigma$ are the Sylow $p$-subgroups of $A_n$, by~\cite[Theorem~B]{Po} (see also \cite[Corollary~4.2]{Ha03}).  The number of ramification points over $\sigma$ is the index of such a Sylow $p$-subgroup $Q$.  Since the index $m$ of $Q$ in $A_n$ is greater than or equal to $t$, the set $\pi^{-1}(\sigma)$ contains at least $t$ points; and hence so does $\pi^{-1}(\Delta)$, which contains it.  
\qed

\begin{prop}\label{spts}
Let $X$ be a smooth connected projective $k$-curve, let $\Sigma \subset X$ be a non-empty finite subset of $X$, and fix a point $\tau$ in $\Sigma$.   
Let $G$ be a finite group and let $\phi: Y \to X$ be a $G$-Galois cover that is \'etale away from $\Sigma$.
Then for every positive integer $s$ there is a cover $\psi_s:X_s \to X$ that is \'etale away from $\Sigma$ and is linearly disjoint from $\phi:Y \to X$ such that the normalized pullback $\phi_s:Y_s \to X_s$  of $\phi$ via $\psi_s$ is a $G$-Galois cover with at least $s$ unbranched points in $\psi_s^{-1}(\tau) \subset X_s$.
\end{prop}

\begin{proof} 
Applying Lemma \ref{sizeSigma} to the cover $\phi: Y \to X$ with $t=2$ and $\Delta = \Sigma$, there is a Galois cover $\pi:X_0 \to X$ of degree $d\geq 2$ that is linearly disjoint from $\phi:Y \to X$, is {\'e}tale away from $\Sigma$, and such that $|\pi^{-1}(\Sigma)| \geq 2$.
By linear disjointness, the normalized pullback $Y_0 \to X_0$ of $Y \to X$ is connected, and so it is a $G$-Galois cover.  Replacing $Y \to X$ by 
$Y_0 \to X_0$ and $\Sigma$ by $\pi^{-1}(\Sigma)$, we may assume that $|\Sigma| \geq 2$.

Next we reduce to the case that $s=1$.  Suppose that there exists a cover $\psi_1:X_1 \to X$ {\'e}tale away from $\Sigma$ and linearly disjoint from $\phi:Y \to X$ such that the normalized pullback $\phi_1:Y_1 \to X_1$  of $\phi$ via $\psi_1$ contains at least one point $\tau_1$ in $\psi_1^{-1}(\tau)$ that is unbranched in $\phi_1$.  Necessarily $Y_1$ is connected by linear disjointness.
Let $\Delta = \psi_1^{-1}(\Sigma) - \{\tau_1\}$.  This is non-empty by the assumption that $|\Sigma| \geq 2$.
Applying Lemma \ref{sizeSigma} again, but this time to the cover $\phi_1$ with $t=s$, there is a cover $\pi_s: X_s \to X_1$ such that its degree is at least $s$; it is linearly disjoint from $\phi_1:Y_1 \to X_1$; it is {\'e}tale away from $\Delta$; and $|\pi_s^{-1}(\Delta)| \ge s$.  Consider the normalized pullback $\phi_s:Y_s \to X_s$ of $\phi_1$ via $\pi_s$; this is a connected $G$-Galois cover because $X_s$ and $Y_1$ are linearly disjoint over $X_1$.
Moreover, as $\tau_1 \not\in \Delta$, the morphism $\pi_s$ is unramified over $\tau_1$.  But $\pi_s$ has degree at least $s$, and so $\pi_s^{-1}(\tau_1)$ contains at least $s$ points.  
Each point in $X_s$ lying over $\tau_1$ is unbranched in $\phi_s$, since $\tau_1$ is unbranched in $\phi_1$.  
Let $\psi_s:= \psi_1\circ \pi_s: X_s \to X$.
Thus $\phi_s:Y_s \to X_s$ is the normalized pullback of $\phi:Y \to X$ via $\psi_s$, and $X_s$ is linearly independent of $Y$ over $X$ by connectivity of $Y_s$.
Finally, since $\tau_1$ lies over $\tau$ in $X_1$, we have that $X_s$ contains at least $s$ points in $\psi_s^{-1}(\tau) \supset \pi_s^{-1}(\tau_1)$ that are unbranched in $\phi_s$.  Thus the cover $\psi_s:X_s \to X$ has the desired properties.

It now remains to prove the result for $s=1$.  We consider three cases.

\smallskip

Case 1. {\sl The inertia groups of $Y \to X$ over $\tau$ are non-trivial $p$-groups, and there is point $\sigma \in \Sigma - \{\tau\}$ such that the inertia groups of $\phi:Y \to X$ over $\sigma$ are the same as those over $\tau$. }  

Pick a point $\eta \in Y$ over $\tau$ and let $P \subset G$ be its inertia group.  According to Proposition~2.7 of \cite{Ha80}, there is a $P$-Galois cover $W \to X$ that is unramified away from $\tau$, totally ramified at $\tau$, and such that the $P$-Galois extensions $\hat{\mathcal O}_{W,\omega}$ and $\hat{\mathcal O}_{Y,\eta}$ of $\hat{\mathcal O}_{X,\tau}$ are isomorphic, where $\omega$ is the unique point of $W$ over $\tau$.  (In fact, according to the result cited, the number of such covers $W \to X$ is equal to the order of $H^1(X,P)$, though we do not need this here.)  

Since $Y \to X$ has the same sets of inertia groups over $\tau$ and $\sigma$, whereas $W \to X$ is totally ramified over $\tau$ and unramified over $\sigma$, it follows that these two Galois covers are linearly disjoint (having no common subcovers).  
So the normalized pullback $\phi_W: Y_W \to W$ of $Y \to X$ via $W \to X$ is an irreducible $G$-Galois cover.  
This cover is unramified at $\omega$, since it is trivial over the fraction field of $\hat{\mathcal  O}_{W,\omega}$ (by the above isomorphism).  Let $X_1 = W$, $Y_1 = Y_W$, and $\phi_1$ be the cover $\phi_W$.
Now, $\omega \in W = X_1$ lies over $\tau \in X$.
Thus $\psi_1:X_1 \to X$ is {\'e}tale away from $\Sigma$ and is linearly disjoint from $\phi:Y \to X$, and the normalized pullback $\phi_1:Y_1 \to X_1$  contains at least one point in $\psi_1^{-1}(\tau)$ that is unbranched in $\phi_1$.

\smallskip

Case 2: {\sl  There are distinct points $\sigma_1, \sigma_2 \in \Sigma - \{\tau\}$ such that the inertia groups of $\phi:Y \to X$ over each $\sigma_i$ are the same as those over $\tau$. }  

Pick a point over $\tau$ and let $I$ be its inertia group.  Then $I$ is a semi-direct product of a $p$-group $P$ and a cyclic prime-to-$p$ group $C = \langle \gamma \rangle$.   The other inertia groups over $\tau$ (and hence over $\sigma_1$ and $\sigma_2$) are conjugate to $I$. Let $g$ be the genus of $X$.  By \cite[XIII, Cor.~2.12]{Gr}, the prime-to-$p$ fundamental group $\pi_1'$ of $X - \{\tau,\sigma_1\}$ has generators $a_i,b_i,c_j$ with $1\le i \le g$ and $j=0,1$, subject to the single relation $(\prod_i [a_i,b_i]) c_0c_1=1$.

Take the surjection $\pi_1' \to C$ given by sending each $a_i$ and $b_i$ to $1$, sending $c_0$ to $\gamma$, and sending $c_1$ to $\gamma^{-1}$.  The corresponding $C$-Galois cover $\psi_T: T \to X$ is branched just at $\tau$ and $\sigma_1$, where it is totally ramified.  Each non-trivial subcover $T' \to X$ of $T \to X$ will also be totally ramified at $\tau$ and $\sigma_1$, and unramified at $\sigma_2$.  On the other hand, each non-trivial subcover of $Y \to X$ will be branched at $\sigma_2$ if it is branched at $\tau$ and $\sigma_1$, because those three points have the same sets of inertia groups.  So $T$ is linearly disjoint from $Y$ over $X$.  Hence the normalized pullback $\phi_T:Y_T \to T$ of $Y \to X$ via $T \to X$ is an irreducible $G$-Galois cover.  Moreover at the unique points $\tau', \sigma_1' \in T$ over $\tau,\sigma_1 \in X$, the inertia groups of $Y_T \to T$ have no tame part, by Abhyankar's Lemma.  That is, the inertia groups over those points are all isomorphic to $P$.  

If $P$ is trivial then the cover $\phi_T:Y_T \to T$ is unbranched over at least one point of $T$ lying over $\tau \in X$.  Taking $\psi_s:X_s \to X$ to be $\psi_T: T \to X$, the normalized pullback $\phi_T:Y_T \to T$ then plays the role of $\phi_s:Y_s \to X_s$, and the result is shown in this case.

On the other hand if $P$ is non-trivial, we apply Case 1 to $\phi_T: Y_T \to T$ with $\Sigma$ replaced by $\psi_T^{-1}(\Sigma)$, $\tau$ replaced by $\tau'$ and $\sigma$ replaced by $\sigma_1'$. 
Doing so we obtain a cover $\psi:X_1 \to T$ that is {\'e}tale away from $\phi_T^{-1}(\Sigma)$, linearly disjoint from $\phi_T:Y_T \to T$, and such that the normalized pullback $\phi_1:Y_1 \to X_1$  of $\phi_T$ via $\psi$ contains at least one point $\tau_1$ in $\psi^{-1}(\tau')$ that is unbranched in $\phi_1$.
The cover $\phi_1:Y_1 \to X_1$ is the normalized pullback of $\phi:Y \to X$ via $\psi_1 = \psi_T \circ \psi$ because $\phi_T$ is the normalized pullback of $\phi$ via $\psi_T$ and $\phi_1$ is the normalized pullback of $\phi_T$ via~$\psi$.  
Similarly, $X_1$ is linearly disjoint from $Y$ over $X$.
Since $\psi_1(\tau_1) = \psi_T \circ \psi(\tau_1) = \psi_T(\tau') = \tau$,
the cover $\psi_1$ has the asserted properties.
 
\smallskip

Case 3: {\sl The general case when $s=1$.}

Applying Lemma \ref{sizeSigma} to the cover $\phi:Y \to X$ with $t=3$ and $\Delta = \{\tau\}$, there exists a cover $\pi:Z \to X$ of degree at least $3$ that is linearly disjoint from $\phi:Y \to X$, is {\'e}tale away from $\Delta$, and such that $|\pi^{-1}(\Delta)| \ge 3$.  Let $Y_Z$ be the normalization of $Y \times_X Z$; this is irreducible because of linear disjointness.  
Hence the induced map $\phi_Z: Y_Z \to Z$ is $G$-Galois.  Moreover, for any two points $\tau_1, \tau_2 \in \pi^{-1}(\tau)$, the sets of inertia groups of $\phi_Z$ over $\tau_1$ and over $\tau_2$ are the same since $\pi$ is Galois.  The result follows by applying Case 2 to the cover $Y_Z \to Z$ with $\Sigma$ replaced by $\pi^{-1}(\Sigma)$, and $\tau, \sigma_1$, and $\sigma_2$ replaced by any three points in $\pi^{-1}(\tau)$. 
\end{proof}

\section{Solving embedding problems with restricted branch locus.}  \label{solve}
The goal of this section is to prove Theorem~\ref{solveEP}, which is
a strengthening of the main theorem of \cite{Ha99}
(Theorem~5.4 of that manuscript) concerning embedding problems for 
branched covers over an algebraically closed field.  That result gave a
bound on the number of additional branch points that are needed in order
to solve a given embedding problem.  In the version we prove below, we 
also show that the positions of the additional branch points can be
specified in advance (unlike \cite[Theorem~5.4]{Ha99}, which allowed
only one of these additional points to be specified).

In order to prove this strengthening, we will build upon the assertion of Theorem~5.4 of \cite{Ha99}, proceeding in two steps. 
First, we allow the additional branch points to move, obtaining a family 
of solutions to the given embedding problem each with the same number of
additional branch points.  Second, we show that the
additional branch points can be specialized to the desired locations,
while preserving the irreducibility and separability of the cover.  The 
assertions needed to carry out these two steps appear in the next two
propositions.

As before, $k$ is an algebraically closed field of characteristic $p>0$.

\medskip

\begin{prop}\label{movepts1}
Let $G$ be a finite group and let
$f:Y \to X$ be a $G$-Galois cover of 
projective $k$-curves with branch locus $B$.  Write $B$ as
a disjoint union $B_0 \cup B_1$ where $B_0 = \{\xi_1,\dots,\xi_r\}$ is a
set of $r > 0$ distinct $k$-points.

\begin{enumerate}

\item[(a)] Then there is an integral affine $k$-variety $T$, a dominating morphism $\pi 
= (\pi_1,\dots,\pi_r):T \to X^r$, and a $G$-Galois cover
$\tilde f:\tilde Y \to X \times T$, such that the fibre over each $\tau
\in T$ is a $G$-Galois cover of $X$ with branch locus
$B_\tau \cup B_1$, where $B_\tau := \{\pi_1(\tau),\dots,\pi_r(\tau)\}$
is disjoint from $B_1$; and such that for some $\tau_0 \in
\pi^{-1}(\xi_1,\dots,\xi_r)$, the fibre over $\tau_0$ is isomorphic to
the given $G$-Galois branched cover $Y \to X$. 

\item[(b)]Moreover, if $H$ is a quotient of $G$ and $Y \to X$ dominates an
$H$-Galois cover $g:Z \to X$ that is unramified over $B_0$,
then the family $\tilde f:\tilde Y \to X \times T$ may be chosen to
dominate the trivial family $\tilde g := g \times {\rm id}:Z \times T
\to X \times T$.
\end{enumerate}
\end{prop}

\begin{proof} At each point $\xi_i$ of $X$, choose a
uniformizing parameter $x_i \in {\mathcal O}_{X,\xi_i}$.  Identify the local 
ring of $X \times X$ at $(\xi_i,\xi_i)$ with the tensor square of ${\mathcal
O}_{X,\xi_i}$ over $k$.  For short, we write $x_i, t_i \in {\mathcal O}_{X
\times X,(\xi_i,\xi_i)}$ for the elements identified with $x_i \otimes 
1, 1 \otimes x_i$ respectively, and we also identify these elements with
their images in the complete local ring
$\hat{\mathcal O}_{X \times X,(\xi_i,\xi_i)}$.  Consider the inclusion of
complete local rings $\iota_i:\hat{\mathcal O}_{X,\xi_i} \hookrightarrow 
\hat{\mathcal O}_{X \times X,(\xi_i,\xi_i)}$ given by $x_i \mapsto x_i-t_i$.
 With respect to $\iota_i$, for each point $\eta \in Y$ lying over
$\xi_i$ we may form the tensor product $A_\eta := \hat{\mathcal O}_{Y,\eta} 
\otimes_{\hat{\mathcal O}_{X,\xi_i}} \hat{\mathcal O}_{X \times
X,(\xi_i,\xi_i)}$; here we regard $\hat{\mathcal O}_{Y,\eta}$ as an
$\hat{\mathcal O}_{X,\xi_i}$-algebra via $f$.  Then $A_i :=
\prod_{f(\eta)=\xi_i} A_\eta$ is a $G$-Galois algebra over $\hat{\mathcal 
O}_{X \times X,(\xi_i,\xi_i)}$ whose spectrum is ramified precisely over
$x_i=t_i$ and whose fibre mod $t_i$ is isomorphic to the pullback of $Y
\to X$ over $\hat{\mathcal O}_{X,\xi_i}$ as a $G$-Galois cover. 
(Here $A_i$ is the induced algebra $\Ind_{G_\eta}^G A_\eta$, where $G_\eta$ is the inertia group at $\eta$.)

Let $R^* = k[[t_1,\dots,t_r]]$, which we identify with the complete
local ring of $X^r$ at the point ${\underline \xi} :=
(\xi_1,\dots,\xi_r) \in X^r$, and whose fraction field will be denoted
by $F^*$.  For each $i$ write $R_i^* = \hat{\mathcal O}_{X \times 
X^r,(\xi_i,{\underline \xi})}$ and let
$A_i^* = A_i \otimes_{\hat{\mathcal O}_{X \times X,(\xi_i,\xi_i)}} R_i^*$.
Since $r > 0$, the curve $U = X - B_0$ is affine and may be written as
$U = {\rm Spec}\, R_U$.  Let $A_U$ be the ring of functions on 
$f^{-1}(U) \subset Y$, and let $A_U^* = A_U \otimes_{R_U} R_U^*$, where
$R_U^* = R_U[[t_1,\dots,t_r]]$.  The mod $(t_1,\dots,t_r)$-reductions of
$A_U^*$ and $A_i^*$ can be identified with $A_U$ and $\hat{\mathcal
O}_{X,x_i}$; and so we obtain isomorphisms between the $G$-Galois
algebras that these reductions induce over $\hat{\mathcal K}_{X,x_i}$, the
fraction field of $\hat{\mathcal O}_{X,x_i}$.  Writing $\hat{\mathcal
K}^*_{X,x_i} = \hat{\mathcal K}_{X,x_i}[[t_1,\dots,t_r]]$, we thus obtain an 
identification $A_U^* \otimes_{R_U^*} \hat{\mathcal K}_{X,x_i}^* = A_i^*
\otimes_{R_i^*} \hat{\mathcal K}_{X,x_i}^*$ of $G$-Galois \'etale algebras,
for each $i$.

Applying the patching result Theorem~3.2.8 of \cite{Ha03}, we obtain a 
$G$-Galois branched cover $Y^* \to X^* := X \times_k R^*$ whose
restrictions to ${\rm Spec}\, R_U^*$ and to ${\rm Spec}\, R_i^*$
respectively agree with ${\rm Spec}\, A_U^*$ and ${\rm Spec}\, A_i^*$,
and whose closed fibre is isomorphic to $Y \to X$.  The cover is 
branched at the points of $B_1^* := B_1 \times_k R^*$ and at $r$
distinct $R^*$-points $\xi_1^*,\dots,\xi_r^*$.  Here $\xi_i^*$ is the
pullback to $X^* = X \times_k {\mathcal O}_{X^r,\underline \xi}$ of the
$X^r$-point $\xi_i^\Delta$ of $X \times X^r$, where $\xi_i^\Delta$ is 
the $X^r$-point at which the first coordinate (in $X$) is equal to
the $i$th entry of the second coordinate.  More formally, if ${\rm
pr}_i:X^r \to X$ is the $i$th projection, then $\xi_i^\Delta$ is the
inverse image of the diagonal $\Delta \subset X \times X$ under ${\rm 
id}_X \times {\rm pr}_i: X \times X^r \to X \times X$.
Also, since the closed fibre $Y$ of $Y^*$ is smooth over $k$, it follows
that $Y^*$ is smooth over $R^*$, because the singular locus is closed
and every closed subset of $Y^*$ must meet the closed fibre $Y$ since $Y$ is projective.  
Note that in the situation of part (b), $Y^* \to X^*$ dominates the
$H$-Galois cover $g^* = g \times {\rm id}_{{\rm Spec}\,R^*}:Z^* := Z
\times_k R^* \to X^*$, since it does over $R_U^*$, $R_i^*$, and
$\hat{\mathcal K}_{X,x_i}^*$, compatibly, and since the patching assertion 
of \cite[Theorem~3.2.8]{Ha03} is an equivalence of categories.

Let $X' = {\rm Spec}\, S$ be an affine open neighborhood of
$\{\xi_1,\dots,\xi_r\}$ in $X$.  We may regard $S^{\otimes r}$ as a
subring of $R^*$.  The $G$-Galois cover $Y^* \to X^*$, being of finite 
type over $R^*$, is induced by a $G$-Galois cover $\tilde Y \to X \times
T$, where $T$ is the spectrum of an $S^{\otimes r}$-algebra $\tilde R$
that is of finite type and is contained in $R^*$.  Thus $\tilde R$ is a 
domain and $T$ is integral.  The inclusion
$S^{\otimes r} \hookrightarrow \tilde R$ corresponds to a morphism $T \to X'^r$; composing this with the embedding $X'^r \hookrightarrow X^r$ yields a morphism $\pi:T \to X^r$.

Possibly after enlarging the choice of
$\tilde R$ by inverting some elements, we may assume that the above
properties of $Y^* \to X^*$ descend to $\tilde Y \to X \times T$.
Specifically, since $Y^*$ is smooth over $R^*$, we may choose $\tilde Y$ 
so as to be smooth over $T$.  Similarly, we may choose $\tilde Y \to X
\times T$ so that its fibre over $X \times \{\tau_0\}$ is isomorphic to
$Y \to X$, where $\tau_0 \in T$ is the image of the closed point of
${\rm Spec}\, R^*$ (i.e.\ $\tau_0$ corresponds to the contraction of the
maximal ideal of $R^*$ under the inclusion $\tilde R \hookrightarrow
R^*$).  Moreover we may choose $\tilde Y \to X \times T$ so that its
branch locus consists of $\tilde B_1 := B_1 \times T$ together with $r$ 
distinct $T$-points $\tilde\xi_i$, where $\tilde\xi_i: T \to X \times T$
is the pullback to $X \times T$, via $\pi:T \to X^r$, of the $X^r$-point
$\xi_i^\Delta$ of $X \times X^r$ defined above.  In addition we may assume
that the loci $\tilde \xi_i$ are pairwise disjoint and are disjoint from 
$\tilde B_1$.
Finally, in part (b), we may assume that $\tilde Y \to X \times T$
dominates the $H$-Galois cover $\tilde g = g \times {\rm id}_T:Z \times
T \to X \times T$.

The morphism $\pi$ is dominating since $S^{\otimes r} \hookrightarrow 
\tilde R$ is an inclusion.
Write $\pi_i = {\rm pr}_i \circ \pi:T \to X$; so
$\pi=(\pi_1,\dots,\pi_r):T \to X^r$.  Let $p:X \times X^r \to X$ denote 
projection onto the first factor~$X$.  Then $p \circ \xi_i^\Delta = {\rm
pr}_i:X^r \to X$.  Also $\xi_i^\Delta \circ \pi = ({\rm id}_X \times
\pi) \circ \tilde \xi_i: T \to X \times X^r$, since $\tilde \xi_i$ is
the pullback of $\xi_i^\Delta$.  Thus $\pi_i = {\rm pr}_i \circ \pi = p
\circ ({\rm id}_X \times \pi) \circ \tilde \xi_i:T \to X$.  So for each
$\tau \in T$, the branch locus of the corresponding fibre of $\tilde Y
\to X \times T$ is branched at $B_1$ and at the $r$ points $\pi_i(\tau)$ 
of $X$ which are respectively the restrictions of the $T$-points $\tilde
\xi_i$ of $X \times T$.  Since $\tilde\xi_i$ pulls back to $\xi_i^*$
under ${\rm Spec}\,R^* \to T$, and since $\xi_i^*$ restricts to $\xi_i$ 
on the closed fibre of $X^* = X \times_k R^*$, we have that
$\pi_i(\tau_0) = \xi_i$.  Finally, the fibres of $\tilde Y \to X \times
T \to T$ are connected by Zariski's Connectedness Theorem (see \cite[III,
Exercise~11.4]{Ht} or \cite[Th\'eor\`eme~4.3.1]{EGA}) since the morphism is
projective and the fibre over $\tau_0$ is the connected curve $Y$.
\end{proof}

\begin{prop}\label{movepts2} 
In the situation of Proposition \ref{movepts1}(b), 
assume that $N := {\rm ker}(G \to H)$ has order prime to $p$.  Let $B' =
\{\xi_1',\dots,\xi_r'\}$ be a set of $r$ distinct $k$-points of $X$ that
is disjoint from $B_1$.  Then there is a $G$-Galois 
cover $f':Y' \to X$ that has branch locus $B' \cup B_1$ and
dominates $g:Z \to X$.
\end{prop}

\begin{proof}  Let $\tilde f:\tilde Y \to X \times T$ be as in
the conclusion of Proposition~\ref{movepts1}(b).  If $\underline\xi' := 
(\xi_1',\dots,\xi_r')$ is in the image of the dominating morphism $\pi:T
\to X^r$, then the conclusion follows from the properties of $\tilde f$
and its fibres.  More generally, we proceed as follows:

Since $T$ is an affine $k$-variety, we may regard $T$ as a closed subset of some $\A^n_k$.  The graph $\Gamma 
\subset T \times X^r \subset \A^n_k \times X^r$ of $\pi$ is isomorphic to
$T$ by the first projection map; so replacing $T$ by $\Gamma$ we may
assume that $T \subset \A^n_k \times X^r$ and that $\pi$ is the second
projection.  Let $\bar T$ be the closure of $T$ in $\P^n_k \times X^r$
and let $\bar \pi:\bar T \to X^r$ be the second projection map.  Since
the projective morphism $\bar \pi$ extends $\pi:T \to X^r$, it is
dominating and hence surjective.  Let $\tau_1 \in \bar T$ be a point 
that lies over $\underline \xi' \in X^r$.  Since $T$ is integral, there is an integral curve $\bar C_0$  in $\bar T$ that passes through $\tau_0$ and $\tau_1$.  Let $C_0 = \bar C_0 \cap T$.  Thus $\bar C_0$ is the closure of $C_0$ in $\bar T$.

The pullback of $\tilde Y \to Z \times T \to X \times T$ under $\bar C_0
\to T$ is a $G$-Galois branched cover $\tilde Y_{\bar C_0} \to Z \times
\bar C_0 \to X \times \bar C_0$ whose restriction to $X \times C_0$ is 
branched only over $B_1 \times C_0$ and at the $C_0$-points
$\xi_{i,C_0}:=(\pi_i,{\rm id}_{C_0}):C_0 \to X \times C_0$ for
$i=1,\dots,r$ (where, as before, $\pi_i:X^r \to X$ is the $i$th
projection).  Let $\bar C$ be the normalization of the curve $\bar C_0$, 
and let $\tilde Y_{\bar C}  \to Z \times \bar C \to X \times \bar C$ be
the normalized pullback of $\tilde Y_{\bar C_0} \to Z \times \bar C_0
\to X \times \bar C_0$ under $\bar C \to \bar C_0$.  So $\tilde Y_{\bar 
C}  \to X \times \bar C$ is unramified away from $B_1 \times \bar C$,
the $\bar C$-points $\xi_{i,\bar C} := (\pi_i,{\rm id}_{\bar C})$, and
$X \times S$, where $S \subset \bar C$ is the finite set of points whose 
image in $\bar C_0$ does not lie in $C_0$.  (Here $\pi_i:\bar C \to X$
denotes the composition of the normalization map $\bar C \to \bar C_0$
with the projection map $\pi_i:\bar C_0 \to X$.)  Hence the $N$-Galois 
cover $\tilde Y_{\bar C}  \to Z \times \bar C$ is unramified away from
$g^{-1}(B_1) \times \bar C$, the pullbacks $(g \times {\rm id}_{\bar
C})^*(\xi_{i,\bar C})$ to $Z \times \bar C$ for $i=1,\dots,r$, and $Z
\times S$.

Let $m$ be the least common multiple of the ramification indices of
$\tilde Y_{\bar C}  \to Z \times \bar C$ over the generic points of $Z
\times S$.  Thus $m$ is prime to $p$, since the order of $N$ is prime to 
$p$.  Choose a cyclic branched cover $\bar C' \to \bar C$ of degree $m$ that is
totally ramified at the points of $S$ (and possibly elsewhere).  Then by
Abhyankar's Lemma, the normalized pullback $\tilde Y_{\bar C'}  \to Z 
\times \bar C'$ of $\tilde Y_{\bar C}  \to Z \times \bar C$ is
unramified at the generic points of $Z \times S'$, where $S' \subset
\bar C'$ is the inverse image of $S \subset \bar C$.  So by Purity of 
Branch Locus \cite[41.1]{Na}, $\tilde Y_{\bar C'}  \to Z \times \bar C'$ is
branched only over $g^{-1}(B_1) \times \bar C'$ and at the pullbacks
$(g \times {\rm id}_{\bar C'})^*(\xi_{i,\bar C'})$ to $Z \times \bar C'$ 
for $i=1,\dots,r$, where $\xi_{i,\bar C'}$ is the pullback of
$\xi_{i,\bar C}$ under $\bar C' \to \bar C$.

Let $\tau' \in \bar C'$ be a point mapping to $\tau_1 \in \bar C_0$ via
$\bar C' \to \bar C \to \bar C_0$.  The components of $B_1 \times \bar 
C'$ and the $\bar C'$-points $\xi_{i,\bar C'}$ are each smooth over
$\bar C'$, as are their inverse images under $Z \times \bar C' \to X
\times \bar C'$.  In the fibre of $X \times \bar C'$ over $\tau'$, these 
loci specialize to distinct points of $X$ (viz.\ to the points of $B_1$
and $B'$ respectively); hence their inverse images in $Z \times \bar C'$
are also disjoint over $\tau'$.  So the union of these inverse images is 
smooth over $\tau'$, and there is an affine neighborhood
$C' \subset \bar C'$ of $\tau'$ such that the (tame) branch locus of the
restriction $\tilde Y_{C'} \to Z \times C'$ of $\tilde Y_{\bar C'}  \to 
Z \times \bar C'$ is smooth over $\bar C'$.  By Abhyankar's Lemma (in
the form \cite[Theorem~2.3.2]{GM}), the cover $\tilde Y_{C'}  \to Z \times
C'$ is a Kummer cover \'etale locally, with smooth branch locus, hence 
is smooth over $C'$.  So the fibre over $\tau'$ is smooth over $k$.
 
The intersection of $C'$ with the inverse image of $C_0$ under $\bar C' \to \bar C_0$ is a dense open subset of $C'$.  For each point $\tau$ in this dense open subset, the fibre of the projective morphism $\tilde Y_{C'} \to C'$ over $\tau$ is connected, by the conclusion of Proposition~\ref{movepts1}.  By Zariski's Connectedness Theorem \cite[III,
Exercise~11.4]{Ht}, it follows that each 
fibre of $\tilde Y_{C'}  \to C'$ is connected, in particular the fibre over $\tau'$.  So the
fibre of the $G$-Galois cover $\tilde Y_{C'} \to X \times C'$ over the 
point $\tau' \in C'$ is a $G$-Galois cover of $X$ with the desired
properties.
\end{proof}

We recall the definition of relative rank of a subgroup of a finite group (see Section 2 of  \cite{Ha99}).
Let $H$ be any finite group  and let $E$ be a subgroup of $H$.
A subset $S \subset E$ will be called a {\it relative generating set} for $E$ in $H$ if for every subset $T \subset H$ such that $E \cup T$  generates $H$, the subset $S \cup T$ also generates $H$.
We define {\it the relative rank of $E$ in $H$} to be the smallest non-negative integer $s:=\rk_H(E)$ such that there is a relative generating set for $E$ in $H$ consisting of $s$ elements.  
Thus every generating set for $E$ is a relative generating set, 
and so $0 \leq \rk_H(E) \leq \rk(E).$
Also, $\rk_H(E) = \rk(E)$ if $E$ is trivial or $E = H$, while $\rk_H(E) = 0$ if and only if $E$ is contained in the Frattini subgroup $\Phi(H)$ of $H$ \cite[p.~122]{Ha99}.

\begin{thm}\label{solveEP}
Let $N$ be a normal subgroup of a finite
group $G$, and let $H = G/N$.  Let $\bar N = N/p(N)$ and $\bar G = 
G/p(N)$, and let $r = \rk_{\bar G}(\bar N)$.   Let $V \to U$ be an
$H$-Galois \'etale cover of affine $k$-curves and let $\xi_1, 
\dots, \xi_r \in U$ be distinct points.
Then there is a $G$-Galois cover $W \to U$ branched only at
$\xi_1, \dots, \xi_r \in U$ such that $W/N \to U$ is isomorphic to $V
\to U$ as an $H$-Galois cover.
\end{thm}

\begin{proof} We may identify $\bar G/\bar N$ with $H$.  Let $X$ be
the smooth completion of $U$ and let $Z$ be the normalization of $X$ in
$V$.  Also let $B_1 = X-U$, let $B_0 = \{\xi_1, \dots, \xi_r\}$, and let 
$B = B_0 \cup B_1$.  We claim that there is a $\bar G$-Galois cover
$\bar Y \to X$ that dominates $Z \to X$ and is branched only at $B$; or
equivalently, a $\bar G$-Galois cover $\bar W \to U$ that dominates $V 
\to U$ and is branched only at $B_0$.

First suppose $r=0$.  For some subgroup $\bar G_0 \subset \bar G$, the cover $H$-Galois cover $V \to U$ is dominated by a $\bar G_0$-Galois \'etale cover $\bar W \to U$, because the fundamental group of the affine $k$-curve $U$ has cohomological dimension $1$ [Se90, Proposition~1] and hence is projective (by \cite[I,~\S3.4 Prop.~16 and \S5.9 Prop.~45]{Se73}).  Thus $\bar G/\bar N = H = \bar G_0/(\bar G_0 \cap \bar N) = \bar G_0 \bar N/\bar N$ and so $\bar G$ is generated by $\bar G_0$ and $\bar N$.  But $\bar N$ is contained in the Frattini subgroup of $\bar G$ because $r=0$.  Hence $\bar G_0 = \bar G$, and $\bar W \to U$ is as desired.

On the other hand if $r \ge 1$, then by \cite[Theorem~5.4]{Ha99} there is a smooth connected
$\bar G$-Galois cover $\bar W_0 \to U$ branched at $r$ points such that
$\bar W_0/\bar N \to U$ is isomorphic to $V \to U$ as an $H$-Galois
cover.  Let $\bar Y_0$ be the normalization of $X$ in $\bar W_0$.  So $Z
= \bar Y_0/\bar N$; $Z \to X$ is unramified outside of $B_1 := X - U$;
and $\bar Y_0 \to X$ is unramified outside of $B_1$ and the set $B'$ consisting of the $r$ branch
points of $Y_0 \to U$.
Since $\bar N$ has order prime to $p$, Proposition~\ref{movepts2} above asserts that
there is a smooth connected $\bar G$-Galois branched cover $\bar Y \to
X$ that is branched only at $B$ and that dominates $Z \to X$. 
This completes the proof of the claim.

Since $p(N)$ is a quasi-$p$ group and $B$ is non-empty, Proposition~4.4
of \cite{Ha03} asserts that the $\bar G$-Galois cover $\bar Y \to X$
is dominated by some smooth connected 
$G$-Galois branched cover $Y \to X$ that is also branched only over $B$. 
Since $\bar Y \to X$ dominates the $H$-Galois cover $Z \to X$, so does $Y \to X$; and hence the restriction $W \to U$ of $Y \to X$ over $U$ has the desired properties.
\end{proof}

\section{Main Theorem} \label{application}

Using the results of the previous sections, we prove our main theorem, that the fundamental group of any affine curve in characteristic $p>0$ is almost $\omega$-free.  To do so we combine the ``newly unbranched'' points in the covers obtained in Section 2 with the main result of Section 3 to solve the induced embedding problem. 

\begin{thm} \label{app1}
Let $k$ be an algebraically closed field of characteristic $p$, let $X$ be a smooth connected projective curve, let $\Sigma \subset X$ be a non-empty set of closed points in $X$, and let $C=X - \Sigma$.  Then $\pi_1(C)$ is almost $\omega$-free.
\end{thm}

\begin{proof}
Explicitly, we wish to show the following.
Let $\mathcal{E} = (\alpha : \pi_1(C) \to G, f: \Gamma \to G)$ be an embedding problem for $\pi_1(C)$.  
Then there exists an open subgroup $H \subset \pi_1(C)$ such that $\alpha_H := \alpha|_H$ is surjective and such that the induced embedding problem $\mathcal{E}_H = (\alpha_H, f)$ has a proper solution.

Let $N  = \ker(f)$, and write $\bar N = N/p(N)$ and $\bar \Gamma = \Gamma/p(N)$.  Let $\tau \in \Sigma$ and let $s= \rk_{\bar \Gamma}(\bar N)$ (see the definition of relative rank just before Theorem~\ref{solveEP}).  Let $\phi:Y \to X$ be the pointed $G$-Galois corresponding to $\alpha$.
By  Proposition \ref{spts} there exists a cover $\psi_s:X_s \to X$ that is \'etale away from $\Sigma$; contains at least $s$ points $\{\tau_1, \tau_2,...,\tau_s \} \subset X_s$ over $\tau \in X$; and whose normalized pullback $\phi_s:Y_s \to X_s$ of $\phi$ over $\psi_s$ is a $G$-Galois cover \'etale away from $\Delta := \psi_s^{-1}(\Sigma) - \{\tau_1, \tau_2,...,\tau_s \}$.
Let $\Sigma_s = \psi_s^{-1}(\Sigma)$.
Applying Theorem \ref{solveEP} to the \'etale loci of $\phi_s:Y_s \to X_s$, there exists a smooth connected $\Gamma$-Galois cover $Z_{s} \to X_{s}$ that is \'etale away from $\Delta \cup \{\tau_1,...,\tau_{s}\} = \Sigma_s$, such that $Z_s/N$ is isomorphic to $Y_s$ as $G$-Galois covers.  Now, $\psi_s: X_s \to X$ is ramified at most over $\Sigma$, and the pullback $\phi_s: Y_s \to X_s$ of $\phi$ via $\psi_s$ is a $G$-Galois cover.  Moreover, the $\Gamma$-Galois cover $Z_s \to X_s$  dominates $Y_s \to X_s$.  Letting $H \subset \pi_1(X -\Sigma)$ and the surjection $\alpha_H: H \to G$ be those that correspond to $\psi_s: X_s \to X$ (after some choice of base point), we find that $\alpha|_H = \alpha_H$ and that the induced embedding problem $\mathcal{E}_H = (\alpha_H, f)$ has a proper solution. 
\end{proof}

The situation in characteristic zero is somewhat different, because fundamental groups of curves in characteristic zero are finitely generated.  Specifically, there is the following result:

\begin{prop}  \label{char0}
Let $X$ be a smooth connected projective curve  of genus $g$ defined over an algebraically closed field $k$ of characteristic $0$ and let $\Sigma$ be a non-empty finite collection of points on $X$. The group $\pi_1(X-\Sigma)$ is almost $\omega$-free if and only if $X-\Sigma$ is not isomorphic to the once or twice punctured projective line.
\end{prop}

\begin{proof} 
Suppose first that $X - \Sigma$ is not isomorphic to the once or twice punctured projective line.  Thus $g \ge 0$, $r := |\Sigma| \ge 1$, and either $g \ge 1$ or $r \ge 3$; hence $2g+r \ge 3$.   Let $\E = (\alpha:\Pi \to G, f:\Gamma \to G)$ be a finite embedding problem for $\pi_1(X-\Sigma)$.  Choose a prime number $d$ that is greater than the order of $\Gamma$.
The surjection $\alpha: \pi_1(X-\Sigma) \to G$ determines a $G$-Galois cover $Y \to X$ that is \'etale away from $\Sigma$.  Pick $\tau \in \Sigma$ and set $\Sigma' = \Sigma - \{\tau\}$.  Thus $X - \Sigma'$ is isomorphic either to the projective line with at least two points deleted, or else to a dense open subset of a curve of genus at least one.  Hence $\pi_1(X - \Sigma')$ has a quotient that is cyclic of order $d$, corresponding to a $d$-cyclic cover $\phi_d: Z_d \to X$ that is \'etale away from $\Sigma'$.  The unramified fibre $\phi_d^{-1}(\tau)$ contains at least $d$ points, and $\Sigma_d :=\phi_d^{-1}(\Sigma)$ contains at least $r-1+d>0$ points.  Since $Z_d$ has genus at least $g$, the fundamental group $\Pi_0 = \pi_1(Z_d -\Sigma_d) \subset \pi_1(X-\Sigma)$ is free of rank at least $2g + (r-1+d) -1 > d$.  Also, $Z_d \to X$ is linearly disjoint from $Y \to X$ because they have no common subcovers, the degree of $Z_d \to X$ being a prime number greater than the degree of $Y \to X$.  Thus $\alpha|_{\Pi_0}:\Pi_0 \to G$ is surjective.  Since the rank of $\Pi_0$ is larger than the minimal number of generators of $\Gamma$, by~\cite[Proposition~17.7.3]{FJ} there is a proper solution to the induced embedding problem $\E_0 = (\alpha|_{\Pi_0}:\Pi_0 \to G, f:\Gamma \to G)$.  Thus $\pi_1(X - \Sigma)$ is almost $\omega$-free.

Conversely, suppose that $X-\Sigma$ is the once or twice punctured projective line over $k$.  In the former case $X-\Sigma$ has no non-trivial \'etale covers, while in the latter case every finite \'etale cover is itself isomorphic to the twice punctured projective line.  Thus any open subgroup $\Pi_0$ of $\Pi = \pi_1(X-\Sigma)$ is either trivial or has rank one.  So given a finite embedding problem $\E = (\alpha:\Pi \to G, f:\Gamma \to G)$ for $\Pi$ such that $\Gamma$ is not cyclic, there is no proper solution to the induced embedding problem $\E_0$ for any open subset $\Pi_0$.  Thus $\pi_1(X - \Sigma)$ is not almost $\omega$-free.
\end{proof}

This result helps clarify the relationships among the conditions on a profinite group being free, $\omega$-free, and almost $\omega$-free.  For infinitely generated profinite groups, every free group is $\omega$-free (i.e.\ every finite embedding problem has a proper solution) but not conversely, as noted in the introduction.  Also, every infinitely generated $\omega$-free  profinite group is (trivially) almost $\omega$-free; and again the converse fails, since the fundamental group of an affine curve in characteristic $p$ is almost $\omega$-free by Theorem~\ref{app1} but is not $\omega$-free, as noted in the introduction.  In contrast, finitely generated free profinite groups $\Pi$ are not $\omega$-free, since an embedding problem $\E = (\Pi \to G, \Gamma \to G)$ cannot have a proper solution if the rank of $\Gamma$ is greater than that of $\Pi$.  Concerning the almost $\omega$-free condition, the above proposition has the following consequence:

\begin{cor} \label{fingen}
A free profinite group of finite rank $m$ is almost $\omega$-free if and only if $m>1$.
\end{cor}

\begin{proof}
A free profinite group of finite rank $m \ge 0$ is isomorphic to the algebraic fundamental group of the complex projective line minus $m+1$ points.  Thus the result follows from the proposition above.
\end{proof}

Note that this corollary could alternatively be proved using that an open subgroup of index $i$ in a free group of finite rank $m$ is itself free of finite rank $1 + i(m-1)$~\cite[Theorem 3.6.2(b)]{riza}.  Proposition~\ref{char0} could then be deduced from that using that the fundamental group of an $r$-punctured curve of genus $g$ (for $r>0$) is free of rank $2g+r-1.$  On the other hand, while the proof given in \cite[Theorem 3.6.2(b)]{riza} is purely group theoretic, that result can alternatively be deduced easily from the Riemann-Hurwitz formula, using fundamental groups of affine curves in characteristic zero, just as Corollary~\ref{fingen} was deduced above from Proposition~\ref{char0}.

\end{document}